\documentclass[leqno,12pt]{article} 
\usepackage{amsmath, amssymb}
\usepackage{amsthm} 
\usepackage[all]{xy}
\usepackage{algorithm, algorithmic}
\theoremstyle{plain} 
\newtheorem{theorem}{\indent\sc Theorem}[section]
\newtheorem{lemma}[theorem]{\indent\sc Lemma}

\newtheorem{proposition}[theorem]{\indent\sc Proposition}

\theoremstyle{definition} 
\newtheorem{definition}[theorem]{\indent\sc Definition}

\newtheorem{example}[theorem]{\indent\sc Example}

\newtheorem*{remark0}{\indent\sc Remark}
\newcommand{\Hom}{\mathrm{Hom}}
\newcommand{\Ext}{\mathrm{Ext}}

\newcommand{\Z}{\mathbb{Z}}
\newcommand{\F}{\mathbb{F}}
\newcommand{\Sq}{\mathrm{Sq}}
\newcommand{\A}{{\mathcal{A}}}

\newcommand{\syz}{\mathrm{Syz}}
\newcommand{\Groebner}{Gr\"{o}bner }
\newcommand{\s}{\mathbf{s}}
\newcommand{\e}{\mathbf{e}}

\newcommand{\lp}{\mathrm{lp}}
\newcommand{\lm}{\mathrm{lm}}
\newcommand{\lc}{\mathrm{lc}}
\newcommand{\lt}{\mathrm{lt}}
\newcommand{\lcm}{\mathrm{lcm}}

\newcommand{\X}{\mathbf{X}}
\newcommand{\Y}{\mathbf{Y}}
\newcommand{\g}{\mathbf{g}}

\newcommand{\Syz}{\mathrm{Syz}}
\newcommand{\Lt}{\mathrm{Lt}}
\renewcommand{\vec}[1]{\mathbf{#1}}
\makeatletter
\def\address#1#2{\begingroup
\noindent\parbox[t]{7.8cm}{%
\small{\scshape\ignorespaces#1}\par\vskip1ex
\noindent\small{\itshape E-mail address}%
\/: #2\par\vskip4ex}\hfill%
\endgroup}%
\makeatother
\title{\uppercase{Projective resolution of modules over the
noncommutative algebra}}
\author{
\bigskip \\
\textsc{Tomohiro Fukaya$^{*}$} 
}
\date{} 
\begin{document}

\maketitle

\footnote{ 
$^{*}$The author is supported by Grant-in-Aid for JSPS Fellows 
(19$\cdot$3177) from Japan Society for the Promotion
of Science.
}

\begin{abstract}
We give an explicit algorithm to compute a projective resolution of a
 module over the noncommutative ring based on the noncommutative
 Gr\"obner bases theory.
\end{abstract}

\section*{Introduction} 
Let $K$ be a field and $\Gamma$ be a ring over $K$. We generally assume
that $\Gamma$ has a unit, $\epsilon\colon K\rightarrow \Gamma$, as well
as an augmentation $\eta \colon \Gamma \rightarrow K$.
For graded $\Gamma$-module $M$ and $N$, $\Ext_\Gamma^{*,*}(M,N)$ is
defined with a projective resolution of $M$. This $\Ext$-functor appears
in various areas. In algebraic topology, it appears as a $E_2$-terms of
the Adams spectral sequence, which is one of the most important tools to
compute homotopy sets. Especially, it is given by
$\Ext_{\A_p}^{*,*}(\F_p,\F_p)$that $E_2$-terms of the spectral sequence
which converges to the stable homotopy groups of the sphere
${}_{(p)}\Pi_S^*$.
Here $\A_p$ denotes the mod $p$ Steenrod algebra and 
\begin{eqnarray*}
  \Pi_S^* &=& \bigoplus_k \lim_{n\rightarrow \infty}[S^{n+k}, S^n]\\ 
  {}_{(p)}\Pi_S^* &=& \Pi_S^* / \{\text{elements of finite order 
  prime to }p\}.
\end{eqnarray*}

The study of the stable homotopy groups of sphere has a long history, so
many tools has been developed. For example, there exists a spectral
sequence converging to $\Ext_{\A_p}^{*,*}(\F_p,\F_p)$.

In this paper, we study an elementary algorithm to compute the
$\Ext$-functor directly from its definition, that is, to compute a
projective resolution of $\Gamma$-modules.


A brief explanation of $\Ext$-functor can be seen in the section 9.2 of 
\cite{MR1793722}.
We fix a positive integer $k$. It is enough to compute $\Ext^{s,t}(M,-)$
for $0 \leq t \leq k$ that we have a sequence of degree-preserving 
$\Gamma$-homomorphisms between graded $\Gamma$ modules
\begin{eqnarray}
\label{eq:proj_resolution_k}
   0 \xleftarrow{} M \xleftarrow{\epsilon} P_0 \xleftarrow{d_0} P_1
  \xleftarrow{d_1} P_2 \xleftarrow{d_2} \cdots \xleftarrow{d_{n-1}} P_n \xleftarrow{d_n}
  \cdots 
\end{eqnarray}
which is exact in degree less than or equal to $k$, and each $P_i$
is a projective $\Gamma$-module. A resolution
(\ref{eq:proj_resolution_k}) is called minimal if 
$d_i(P_{i+1}) \subset I(\Gamma)\cdot P_i$ for all $i\geq 0$ where
$I(\Gamma) = \ker(\eta\colon \Gamma \rightarrow K)$
\begin{proposition}
 Let $0 \xleftarrow{} M \xleftarrow{\epsilon} P_0 \xleftarrow{d_0} P_1
  \xleftarrow{d_1} P_2 \xleftarrow{d_2} \cdots \xleftarrow{d_{n-1}} P_n
 \xleftarrow{d_n} \cdots $ be a minimal resolution of $M$ by projective
 $\Gamma$ modules. Then $\Ext_\Gamma^{s}(M,K)\cong \Hom(p_s,K)$.
\end{proposition}
We study the explicit algorithm to compute such a
resolution. 

\begin{definition}
 Let $P$ be a graded $\Gamma$-module. $\{\g_1,\dots,\g_N\}\subset P$ 
is $(\Gamma,k)$-generating set of $P$ if inclusion
\[
 \langle \g_1,\dots, \g_N \rangle \hookrightarrow P
\]
induces an isomorphism of $K$-vector spaces in degree less than or
 equal to $k$.
\end{definition}
Suppose that we have a part of the resolution
(\ref{eq:proj_resolution_k})
\begin{eqnarray}
\label{eq:part_of_resolution}
   P_{n-1} \xleftarrow{d_{n-1}} P_n 
   \xleftarrow{d_n} P_{n+1}.
\end{eqnarray}
We also suppose that we have a $\Gamma$-basis $\{\e_1, \dots, \e_N\}$ 
(resp. $\{\e'_1,\dots, \e'_M\}$) of $P_{n}$ (resp. $P_{n+1}$).
That is,
\[
 P_{n} \cong \bigoplus_{i=1}^N \Gamma\e_i \text{ and } 
 P_{n+1} \cong \bigoplus_{i=1}^M \Gamma\e'_i.
\]
We assume that $\{d_n(\e'_1),\dots,d_n(\e'_M)\}$ is minimal generating
set in the sense of Definition~\ref{def:minimal_generating_set}.
 To extend
the resolution (\ref{eq:part_of_resolution}), we need to know an
$(\Gamma,k)$-generating set $\{\vec{h}_1,\dots, \vec{h}_L\}$ of 
\[
 \ker\left[d_n\colon \bigoplus_{i=1}^{M} \Gamma\e'_i 
            \rightarrow \bigoplus_{i=j}^{N} \Gamma\e_j\right].
\]
Section~\ref{sec:proj-resol-over} gives us an algorithm to compute it.
Especially, we can choose $\{\vec{h}_1,\dots, \vec{h}_L\}$ as a minimal
generating set in the sense of
Definition~\ref{def:minimal_generating_set}.

Set $P_{n+2} = \bigoplus_{i=1}^L \Gamma\e''_i, |\e''_i| = |\vec{h}_i|$
 for $i = 1,\dots,L$. We define the
$\Gamma$-homomorphism
\[
 d_{n+1}\colon \bigoplus_{i=1}^{L} \Gamma\e''_i 
            \rightarrow \bigoplus_{i=1}^{M} \Gamma\e'_i
\]
by $d_{n+1}(\e''_i) = \vec{h}_i$.
Then we have a longer resolution
\begin{eqnarray*}
   P_{n-1} \xleftarrow{d_{n-1}} P_n 
   \xleftarrow{d_n} P_{n+1}
   \xleftarrow{d_{n+1}} P_{n+2}
\end{eqnarray*}
which is minimal and exact in degree less than or equal to $k$ 
(see Proposition~\ref{prop:minimal}). Thus we
have $\Ext^{n+1,t}_\Gamma(M,K) \cong \Hom^t_\Gamma(P_{n+1},K)$ for 
$t \leq k$.
\section{Noncommutative ring}
The most important theory in computational algebra is Gr\"obner basis
theory. In this paper, we are interested in noncommutative algebras. Thus
we need to construct noncommutative Gr\"obner basis theory. It  is now
used in various fields. Such fields are listed in 
the introduction of \cite{MR1947291}. In this section, we give the 
noncommutative Gr\"obner basis theory following the commutative theory given
in \cite{MR1287608}.
\subsection{Monomial order and division algorithm}
Let $K$ be a field and 
$R = K \langle x_1, x_2, \cdots \rangle$ be a free non commutative
graded $K$-algebra with grading $|x_i| = \deg x_i = a_i > 0$. 
We call an element 
$X := x_{i_1}x_{i_2}\dots x_{i_n}$ of $R$ a {\it monomial} of
$R$. Similarly, we call an element $cX$ of $R$ a {\it term} of $R$ where
$c \in K$ and $X$ is a monomial of $R$.
\begin{definition}
 Let $<$ be a well-ordering on the set of the monomials of
 $R$. $<$ is a {\it monomial ordering} of $R$ if the following conditions are
 satisfied.
   \begin{enumerate}
    \item If $U,V,X,Y$ are monomials with $X < Y$, then $UXV < UYV$.
    \item For monomials $X$ and $Y$, if $X = UYV$ for some monomials
	  $U,V$ with $U \neq 1$ or $V \neq 1$, 
	  then $Y < X$. (Hence $1 < X$ for all monomial 
	  $X$ with $X \neq 1$.)
   \end{enumerate}
\end{definition}
\begin{remark0}
 If $R$ is a commutative ring $K[x_1,\dots,x_n]$, then any ordering on
 the set of monomial which satisfies above two condition is a
 well-ordering. On the contrary, if $R$ is a noncommutative ring,
 being a well-ordering does not follow from the above two condition.
\end{remark0}
\begin{example}
\label{ex:lex-order}
We define a monomial ordering on $R$ as
 follows.
Let $X = x_{a_1}\cdots x_{a_k}$ and $Y = x_{b_1}\cdots x_{b_l}$ 
be monomials of $R$.
Then, $X \geq Y$ if $k > l$ or,
$k = l$ and the {\it left-most} nonzero entry of $(a_1 - b_1,\dots,a_k - b_k)$
is positive.
\end{example}
In this section, we fix a monomial ordering $<$ of $R$.
For a non-zero polynomial $f \in R$, we may write f as a linear
combination of monomials of $R$. That is,
\[
 f = c_1X_1 + \dots + c_mX_m
\]
where $c_i \in K \setminus \{0\}$, 
$X_i$ is a monomial satisfying  $X_1 > X_2 > \dots > X_m $. 
We define:
\begin{itemize}
 \item $\lm(f) = X_1$, the {\itshape leading monomial} of $f$; 
 \item $\lc(f) = a_1$, the {\itshape leading coefficient}  of $f$; 
 \item $\lt(f) = a_1X_1$, the {\itshape leading term} of $f$.
\end{itemize}
We also define $\lp(0) = \lc(0) = \lt(0) = 0$.

\begin{definition}
 Let $X$ and $Y$ be monomials of $R$. $X$ is
 divisible by $Y$ if there exists monomials $U,V$ of $R$ such that $X = UYV$.
\end{definition}

Let $(f_1,\dots,f_n)$ be ordered $n$-tuple of $R$. We study the 
{\it division} of $f \in R$ by $(f_1,\dots,f_n)$. First we consider the
special case of the division of $f$ by $g$, where $f,g \in R$.
\begin{definition}
 Given $f,g,h$ in $R$  with $g \neq 0$, we says that {\it $f$ reduces to
 $h$ modulo $g$ in one step}, written
\[
 f \xrightarrow{g} h,
\]
if $\lm(g)$ divides a non-zero monomial $X_i$ that appears in f and 
\[
 h = f - \frac{c_iX_i}{\lt(g)} g.
\]
\end{definition}
\begin{example}
Set  $f = x_1x_2x_3x_1 + x_1x_2, g = x_1x_2 + x_2\in R$. Also, let the
 order be that defined in Example \ref{ex:lex-order}. Then
\begin{eqnarray*}
 f \xrightarrow{g} -x_2x_3x_1 + x_1x_2 \xrightarrow{g} -x_2x_3x_1 - x_2.
\end{eqnarray*}
\end{example}
We extend the process defined above to more general setting.
\begin{definition}
 Let $f,h$
 be polynomials in $R$
 and $F = \{f_\lambda\}_{\lambda \in \Lambda}$ be a family of non-zero
 polynomials. We say that
 {\it $f$ reduces to $h$
 modulo $F$, denoted}
\[
 f \xrightarrow{F}_+ h,
\]
if there exists a sequences of indices 
$\lambda_1,\lambda_2,\dots, \lambda_t \in \Lambda$ 
and a sequences of polynomials
 $h_1, \dots, h_t \in R$ such that
\begin{eqnarray*}
 f \xrightarrow{f_{\lambda_1}} h_1 \xrightarrow{f_{\lambda_2}} h_2
  \xrightarrow{f_{\lambda_3}} \cdots \xrightarrow{f_{\lambda_{t-1}}} h_{t-1}
  \xrightarrow{f_{\lambda_t}} h_t = h.
\end{eqnarray*}
\end{definition}
\begin{definition}
 A polynomial $r \in R$ is called reduced with respect to 
$F=\{f_\lambda\}_{\lambda \in \Lambda}$ if $r=0$ or no monomial that
 appears in $r$ is divisible by any one of the $\lm(f_\lambda),\lambda
 \in \Lambda$. If $f\xrightarrow{F}_+ r$ and $r$ is reduced with respect
 to $F$, then we call $r$ a remainder for $f$ with respect to $F$.
\end{definition}
\begin{proposition} 
\label{prop:division}
Let $f$ be a polynomial in $R$ and Let 
$F = \{f_\lambda\}_{\lambda \in \Lambda}$ be a family of non-zero
 polynomials in $R$. Then $f$ can be written as 
\[
 f = \sum_{i=1}^{t}c_ip_if_{\lambda_i}q_i +r
\]
where $c_i \in K, p_i,q_i,r \in R, i=1,\dots, t$ such that $r$ is
 reduced with respect to $F$ and 
\[
 \lm(f) = \max\left\{\max_{1\leq i \leq t}\lm(p_if_{\lambda_i}q_i),
 \lm(r)
 \right\}.
\]
\end{proposition} 
Unfortunately, this decomposition depends on the choice of the order of
$F=\{f_\lambda\}_{\lambda\in\Lambda}$,
and $f\in \langle f_\lambda | \lambda \in \Lambda \rangle$
does not imply $f\xrightarrow{F}_+ 0$. Here $\langle f_\lambda | \lambda
\in \Lambda \rangle$ denotes the two-side ideal generated by
$F=\{f_\lambda\}_{\lambda\in\Lambda}$.
We can overcome this difficulty of remainders 
by choosing a Gr\"obner basis defined as:
\begin{definition}
 A family of non-zero polynomials 
$G=\{g_\lambda\}_{\lambda \in \Lambda}$ contained in a
 two-side ideal $I$ of $R$, is a called {\it Gr\"obner basis} for $I$ if
 for all $f \in I \setminus \{0\}$, there exists 
$\lambda \in \Lambda$ satisfying that $\lp(g_\lambda)$ divides $\lp(f)$.
\end{definition}
\begin{remark0}
 Since $R$ is not Noetherian ring, Gr\"obner basis is not necessarily
 a finite set.
\end{remark0}
\begin{theorem}
\label{th:equivalent_conditions}
 Let $I$ be a non-zero two-side ideal of $R$. The following statements
 are equivalent for a set of non-zero polynomials
 $G=\{g_\lambda\}_{\lambda \in \Lambda} \subset I$.
\begin{enumerate}
 \item \label{item:Grobner}
       $G$ is a Gr\"obner basis for $I$.
 \item \label{item:remainder}
       $f \in I$ if and only if $f \xrightarrow{G}_+ 0$.
 \item \label{item:division}
       $f \in I$ if and only if $f = \sum_{i = 1}^{t}u_ig_{\lambda_i}v_i$
       with 
       $\lp(f) = 
       \max_{1 \leq i \leq t}(\lp(u_i)\lp(g_{\lambda_i})\lp(v_i))$ where
       $u_i,v_i \in R$.
 \item \label{item:basis}
       A basis of the space $R/I$ consists
       of the coset of all the monomials $X$ in $R$ which is reduced
       with respect to $G$.
\end{enumerate} 
\end{theorem}
\begin{proof}
$(\ref{item:basis} \Rightarrow \ref{item:remainder})$
 Set $f \in R$. By Proposition~\ref{prop:division}, 
 there exists a reduced polynomial $r\in R$ 
 such that $f \xrightarrow{G}_+ r$. Then $r$ is a linear combination of
 reduced monomials with respect to $R$.
 Since $f$ and $r$ represent the same element in $R/I$,
 $f$ is zero in $R/I$ if and only if $r = 0$ in $R$.

For the remaining part of the proof, 
 see \cite[Theorem 1.6.2 and Proposition 2.1.6]{MR1287608}.
\end{proof}
Let $G=\{g_\lambda\}_{\lambda\in\Lambda}$ be a Gr\"obner basis.
For any $f\in R$, let $r$ be a reduced polynomial in $R$ such that 
$f \xrightarrow{G}_+ r$. Theorem~\ref{th:equivalent_conditions} implies
that $r$ is unique. 
We call $r$ the normal form of $f$ with respect to $G$, denoted $N_G(f)$.
\subsection{S-polynomials}
The key to compute a Gr\"obner basis is the S-polynomial. In the
contrast to the commutative case, the S-polynomial of noncommutative
polynomials $f$ and $g$ in $R$ is not unique.
\begin{definition}
 Let $f,g$ be polynomials in $R$. Set $\lm(f) = x_{i_1}\dots x_{i_n}$
 and $\lm(g) = x_{j_1}\dots x_{j_m}$. We define the coefficient set of
 the S-polynomial of $f$ and $g$, denoted $C(f,g)$ by
\begin{eqnarray*}
 C(f,g) &=& \left\{(x_{a_1}\dots x_{a_\alpha},1,x_{c_1}\dots x_{c_\gamma})\in
  R^3\left| 
\begin{array}{c}
  (a_1,\dots,a_\alpha) = (j_1,\dots,j_\alpha), \\
  (i_1,\dots,i_{m-\alpha}) = (j_{\alpha+1},\dots,j_m)\\
  (i_{m-\alpha+1},\dots,i_n) = (c_1,\dots,c_\gamma)
\end{array}  
     \right.
    \right\}\\
&\cup&
 \left\{(1,x_{b_1}\dots x_{b_\beta},x_{c_1}\dots x_{c_\gamma})\in
  R^3\left| 
\begin{array}{c}
  (i_1,\dots,i_\beta) = (b_1,\dots,b_\beta), \\
  (i_{\beta+1},\dots,i_{\beta+m}) = (j_1,\dots,j_m)\\
  (i_{\beta+m+1},\dots,i_n) = (c_1,\dots,c_\gamma)
\end{array}  
     \right.
 \right\}.
\end{eqnarray*}

\end{definition}
Let $f,g$ be polynomials and $(z,p,q) \in C(f,g)$.
We define the {\it S-polynomial} of
$f$ and $g$ associated with $(z,p,q)$, denoted $S(f,g;z,p,q)$, by
\[
 S(f,g;z,p,q) = zf - pgq.
\]
\begin{theorem}[Buchberger's theorem for noncommutative algebra]
\label{th:Buchberger}
 Let $G = \{g_\lambda\}_{\lambda\in \Lambda}$ be a family of non-zero
 polynomials in $R$. Then $G$ is a Gr\"obner basis for the ideal 
$I =\langle g_\lambda | \lambda \in \Lambda \rangle$
 if and only if for all
 $\lambda, \gamma \in \Lambda$, and for all 
$(z,p,q) \in C(g_\lambda,g_\gamma)$,
\[
 S(g_\lambda, g_{\gamma};z,p,q) \xrightarrow{G}_+ 0.
\]
\end{theorem}
In \cite{MR1947291}, there are no explicit description of the proof of
Buchberger's theorem for noncommutative algebra. We give the proof
according to the proof for commutative case by \cite{MR1287608}.
Before we can prove this result, we need one preliminary lemma.
\begin{lemma}
\label{lem:linear_combination}
Let $f_1,\dots,f_s \in R$ be polynomials such that $\lp(f_i) = X \neq 0$
 for all $i = 1,\dots, s$. Let $f = \sum_{i=1}^{s}c_if_i$ with $c_i \in
 K$, $i=1,\dots, s$. If $\lp(f) < X$, then $f$ is a linear combination
 with coefficients in $K$, of $S(f_i,f_j;1,1,1)$, $1\leq i,j \leq s$.
\end{lemma}
\begin{proof}
 See the proof of \cite[Lemma 1.7.5]{MR1287608}.
\end{proof}
\begin{proof}[Proof of Theorem \ref{th:Buchberger}]
If $G  = \{g_\lambda\}_{\lambda\in \Lambda}$ is a Gr\"obner basis of 
$I = \langle g_\lambda | \lambda \in \Lambda \rangle$, then 
$ S(g_\lambda, g_{\gamma};z,p,q) \xrightarrow{G}_+ 0$ by Theorem
 \ref{th:equivalent_conditions}, since 
$S(g_\lambda, g_{\gamma};z,p,q) \in I$.

Conversely, let us assume that 
$S(g_\lambda, g_{\gamma};z,p,q) \xrightarrow{G}_+ 0$ for all 
$\lambda \neq \gamma \in \Lambda, (z,p,q) \in C(g_\lambda, g_\gamma)$. 
We will use Theorem~\ref{th:equivalent_conditions}, \ref{item:division}
 to show
 that $G$ is a Gr\"obner basis for $I$. Set $f\in I$. Then $f$ can be
 written in many ways as linear combination of the $g_\lambda$'s. We
 choose to write $f = \sum_{i=1}^{t}u_{i}g_{\lambda_i}v_{i}$,
 $u_i, v_i \in R$, with
\[
 X = \max_{1\leq i \leq t}(\lp(u_i)\lp(g_{\lambda_i})\lp(v_i))
\]
least.
If $X = \lp(f)$, we are done. Otherwise, $\lp(f) < X$. We will find the
 representation of $f$ with a smaller $X$. 
Let $S = \{i|\lp(u_i)\lp(g_{\lambda_i})\lp(v_i) = X\}$. For $i
 \in S$, write $u_i = c_iX_i + \text{lower terms}$ and 
 $v_i = Y_i + \text{lower terms}$.
 Set $g = \sum_{i\in S}c_iX_ig_iY_i$. Then,
 $\lp(X_ig_{\lambda_i}Y_i) = X$, for all $i \in S$, but 
 $\lp(g) < X$. By lemma~\ref{lem:linear_combination}, there exists
 $d_{i,j} \in K$ such that
\[
 g = \sum_{i,j \in S, i\neq j}
      d_{i,j}S(X_ig_{\lambda_i}Y_i,X_jg_{\lambda_j}Y_j;1,1,1).
\]
Now, by the definition of $S$, for $i,j\in S$ we have
\begin{eqnarray*}
 X &=&
       X_ig_{\lambda_i}Y_i\\
   &=&
       X_jg_{\lambda_j}Y_j\\
   &=&
    \lcm(X_ig_{\lambda_i}Y_i,X_jg_{\lambda_j}Y_j),
\end{eqnarray*}
so it follows that,
\begin{eqnarray*}
 S(X_ig_{\lambda_i}Y_i,X_jg_{\lambda_j}Y_j;1,1,1) &=&
       X_ig_{\lambda_i}Y_i
     -
       X_jg_{\lambda_j}Y_j\\
&=& U_{ij}S(g_{\lambda_i},g_{\lambda_j};z,p,q)V_{ij}
\end{eqnarray*}
for some monomials $U_{ij},V_{ij}$ in $R$ and
 $(z,p,q)\in C(g_{\lambda_i},g_{\lambda_j})$.
By the hypothesis, $S(g_{\lambda_i},g_{\lambda_j};z,p,q) \xrightarrow{G}_+ 0$,
 thus 
$S(X_ig_{\lambda_i}Y_i,X_jg_{\lambda_j}Y_j;1,1,1) \xrightarrow{G}_+ 0$. 
This gives a presentation
\[
 S(X_ig_{\lambda_i}Y_i,X_jg_{\lambda_j}Y_j;X) =
 \sum_{\nu = 1}^{s}u_{ij\nu}g_{\lambda_nu}v_{ij\nu},
\]
such that, by Proposition~\ref{prop:division},
\begin{eqnarray*}
 \max_{1\leq \nu \leq s}(\lp(u_{ij\nu})\lp(g_{\lambda_nu})\lp(v_{ij\nu}))
 &=& \lp(S(X_ig_{\lambda_i}Y_i,X_jg_{\lambda_j}Y_j;1,1,1))\\
 &<& \max(\lp(X_ig_{\lambda_i}Y_i),\lp(X_jg_{\lambda_j}Y_j))) = X.
\end{eqnarray*}
Substituting these expressions into $g$ above, and $g$ into $f$, 
we get a desired contradiction.
\end{proof}

\section{Modules over a noncommutative ring}
\subsection{Gr\"obner basis for modules}
Let $K$ and $R$ be as above.
Let $\e_1,\dots,\e_m$ be standard basis of $R^m$, 
\[
 \e_1 = (1,0,\dots,0), \e_2 = (0,1,\dots,0),\dots,  \e_m = (0,0,\dots,0,1).
\]
 Then, by a {\itshape monomial} in $R^m$ we mean a vector of the type
 $X\e_i$ $(1\leq i \leq m)$, where $X$ is a monomial in $R$.
\begin{example}
 $(0,x_1x_3x_1,0), (0,0,x_2)$ are monomials in $R^3$ but $(x_1,x_2,0)$ is not.
\end{example}
Similarly, by a {\itshape term}, we mean a vector of the type
$c\mathbf{X}$, where $c \in K \setminus \{0\}$ and $\mathbf{X}$ is a
monomial.
\begin{example}
 $(0,5x_1x_1x_3,0,0)$ is a term of $R^4$ but not a monomial. 
\end{example}
If $\vec{X} = cX\e_i$ and $\vec{Y} = dY\e_j$ are terms of
$R^m$, we say that $\mathbf{X}$ divides $\vec{Y}$ provided $i = j$
and there is a monomial $Z$ in $R$ such that $Y = ZX$. We write
\[
 \frac{\vec{Y}}{\vec{X}} = \frac{dY}{cX} = \frac{d}{c}Z.
\]
\begin{example}
 $(0,x_1x_3,0)$ divides $(0,x_1x_1x_3,0)$ but does not divide $(0,x_3,0)$ or
 $(x_2x_1x_3,0,0)$, so we have
\[
 \frac{(0,x_1x_1x_3,0)}{(0,x_1x_3,0)} = \frac{x_1x_1x_3}{x_1x_3} = x_1.
\]
\end{example}
If there exists a monomial $Z \in R$ such that $\X = Z\Y$ or $\Y = Z\X$,
 then we define the least common multiple of $\X$ and $\Y$, denoted 
$\lcm(\X,\Y)$, by
\[
 \lcm(\X,\Y) = \begin{cases}
		Z\Y & \text{ if } \X = Z\Y\\
		Z\X & \text{ if } \Y = Z\X.
	       \end{cases}
\]
Otherwise we define  $\lcm(\X,\Y) = 0$.
 \begin{definition}
  By a term order on the monomials of $R^m$, we mean a well-ordering $<$
  on these monomials satisfying the following conditions:
\begin{enumerate}
 \item If $\mathbf{X} < \mathbf{Y}$ then $Z\mathbf{X} < Z\mathbf{Y}$ for
       all monomials $\mathbf{X},\mathbf{Y}$ and every monomial $Z$ in
       $R$.
 \item $\mathbf{X} < Z\mathbf{X}$ , for every monomial $\mathbf{X}$ in
       $R^m$ and monomial $Z \neq 1$ in $R$. 
\end{enumerate}
 \end{definition}
We define an order on $R^m$ using an order of $R$.
\begin{definition}[POT (position over term)]
 For monomials 	$\mathbf{X}=X\e_i, \mathbf{Y}=Y\e_j \in R^m$, where $X,Y$ are
 monomials in $R$, we say that
   \begin{eqnarray*}
  \mathbf{X} < \mathbf{Y} \text{ if and only if }
   \begin{cases}
     &i > j \\
     &\text{ or }\\ 
     &i=j \text{ and } X < Y.
   \end{cases}
 \end{eqnarray*}
\end{definition}
We now adopt some notation. We first fix a term order $<$ on the
monomials of $R^m$. Then for all $\mathbf{f}\in R^m$, with
$\mathbf{f}\neq 0$, we may
write
\[
 \mathbf{f} = a_1\mathbf{X}_1 + a_2\mathbf{X}_2 + \dots +
 a_i\mathbf{X}_i+ \dots + a_r\mathbf{X}_r,
\]
where, for $1\leq i \leq r$, $a_i \in K \setminus \{0\}$ and
$\mathbf{X}_i$ is a monomial in $R^m$ satisfying $\mathbf{X}_1 >
\mathbf{X}_2 > \cdots > \mathbf{X}_r$. We define
\begin{itemize}
 \item $\lm(\mathbf{f}) = \mathbf{X}_1$, the leading monomial of
       $\mathbf{f}$.
 \item $\lc(\mathbf{f}) = a_1$, the leading coefficient of
       $\mathbf{f}$.
 \item $\lt(\mathbf{f}) = a_1\mathbf{X}_1$, the leading term of
       $\mathbf{f}$.
\end{itemize}
We define $\lm(\mathbf{0}) = \lt(\mathbf{0}) = \mathbf{0}$ and 
$\lc(\mathbf{0}) = 0$.
\begin{remark0}
 $\lm,\lc$ and $\lt$ are multiplicative in the following sense:
$\lm(f\g) = \lp(f)\lm(\g)$, $\lc(f\g) 
= \lc(f)\lc(\g)$ and 
$\lt(fg) = \lt(f)\lt(\g)$, for all $f\in R$ and $\g \in R^m$.
\end{remark0}
Similar to the reduction of polynomials, we introduce the reduction of
vectors.
\begin{definition}
 Given vectors $\vec{f},\vec{g},\vec{h}$ in $R^m$ with 
 $\vec{g} \neq 0$, we says that 
 {\it $\vec{f}$ reduces to $\vec{h}$ modulo $\vec{g}$ in one step}, written
\[
 \vec{f} \xrightarrow{\vec{g}} \vec{h},
\]
if $\lt(\vec{g})$ divides a non-zero term $\X_i$ that appears in
 $\vec{f}$ and 
\[
 \vec{h} = \vec{f} - \frac{\X_i}{\lt(\vec{g})} \vec{g}.
\]
\end{definition}
\begin{example}
Set  $\vec{f} = (x_1x_2x_3x_1 + x_1x_2,x_1,0), 
 \vec{g} = (x_1x_2 + x_2,0,x_3)\in R^3$. Also, let the
 order be POT. Then,
\begin{eqnarray*}
 \vec{f} \xrightarrow{\vec{g}} (-x_2x_3x_1 + x_1x_2,x_1,-x_3x_3x_1) 
\xrightarrow{\vec{g}} (x_2x_3x_1 + x_2, x_1, x_3x_3x_1 + x_3).
\end{eqnarray*}
\end{example}
Let $\Omega = \{\omega_\lambda\}_{\lambda\in \Lambda}$ be a subset of
$R$ and $\langle \Omega \rangle$ denote the two-side ideal of $R$ 
generated by $\Omega$. We suppose that $\Omega$ is a Gr\"obner basis.
For a positive integer $k$, set 
$\Omega(k) = \Omega \cup \{g \in R| |g| > k\}$. Then $\Omega(k)$ is also 
a Gr\"obner basis.
We define a map 
$N_{\Omega,k}\colon \bigoplus_{i=1}^mR\e_i 
           \rightarrow \bigoplus_{i=1}^mR\e_i$ 
by 
\[
  N_{\Omega,k}(f_1,\dots,f_s) 
    = (N_{\Omega(k-|\e_1|)}(f_1),\dots, N_{\Omega(k-|\e_m|)}(f_m)).
\]

\begin{definition}
 Let $\vec{f},\vec{h}$ and $\vec{f}_1,\dots,\vec{f}_s$ be vectors in $R^m$ 
 with $\vec{f}_i \neq 0$,
 and set $F = \{\vec{f}_1, \dots, \vec{f}_s\}$. We say that 
 $\vec{f}$ $(\Omega,k)$-{\it reduces to} $\vec{h}$
 modulo $F$, denoted
\[
 \vec{f} \xrightarrow{F}_{\Omega,k} \vec{h},
\]
if there exists a sequences of indices 
$i_1,i_2,\dots, i_t \in \{1,\dots,s\}$ and a sequences of vectors
 $\vec{h}_1, \dots, \vec{h}_t \in R$ such that
\begin{eqnarray*}
 \vec{f} \xrightarrow{\vec{f}_{i_1}} \vec{h}_1 
         \xrightarrow{\vec{f}_{i_2}} \vec{h}_2
         \xrightarrow{\vec{f}_{i_3}} \cdots 
         \xrightarrow{\vec{f}_{i_{t-1}}} \vec{h}_{t-1}
         \xrightarrow{\vec{f}_{i_t}} \vec{h}_t.
\end{eqnarray*}
and $N_{\Omega,k}(\vec{h}_t) = \vec{h}$.
\end{definition}
 A vector $\vec{r} \in R$ is called 
$(\Omega,k)$-{\it reduced with respect to} 
$F=\{\vec{f}_1,\dots,\vec{f}_t\}$ if $\vec{r}=0$ or no monomial that
 appears in $\vec{r}$ is divisible by any one of the 
$\lm(\vec{f}_i), i \in \{1,\dots,t\}$ and $N_{\Omega,k}(\vec{r}) = \vec{r}$.

\subsection{$(\Omega,k)$-Gr\"obner basis}
\begin{definition}
 Let $\vec{f}_1,\dots,\vec{f}_s$ be vectors in
 $\bigoplus_{i=1}^mR\e_i$. A submodule $M$ of $\bigoplus_{i=1}^mR\e_i$
 is $(\Omega,k)$-generated by $F = \{\vec{f}_1,\dots,\vec{f}_s\}$, denoted 
 $M = \langle F \rangle_{\Omega,k} = 
 \langle \vec{f}_1,\dots,\vec{f}_s \rangle_{\Omega,k}$, if 
 \[
  M = \left\{\left.
             \sum_{i=1}^{t}p_i\vec{f}_{i} + \sum_{i=1}^{m}q_i\e_i
      \right|
             p_i \in R, q_i \in \langle \Omega(k-|\e_i|) \rangle
      \right\}.
 \]
\end{definition}
\begin{definition}
  A set of non-zero vectors
 $G=\{\vec{g}_1,\dots,\vec{g}_t\} \subset \bigoplus_{i=1}^mR\e_i$, 
 is a called $(\Omega,k)$-{\it Gr\"obner basis} for 
 $M = \langle \vec{g}_1,\dots,\vec{g}_t \rangle_{\Omega,k}$ if
 for all $\vec{f} \in M$ such that $\vec{f} \neq 0$, 
one of the following two conditions is satisfied.
\begin{enumerate}
 \item  There exists 
 $i \in \{1,\dots, t\}$ satisfying that 
 $\lm(\vec{g}_i)$ divides $\lm(\vec{f})$.
 \item There exists $j \in \{1,\dots,m\}$ and 
       $q\in \langle (\Omega(k-|\e_j|))\rangle$ such that 
       $\lm(q\e_j)$ divides $\lm(\vec{f})$.
\end{enumerate}
Here $\langle \Lt(\Omega(k-|\e_i|))\rangle$ denotes the two side ideal
 generated by $\{\lm(g)\e_i|g \in \Omega(k-|\e_i|)\}$.
\end{definition}

\begin{proposition} 
\label{prop:division_grobner_module}
 Let $\vec{f}$ be a vector in $\bigoplus_{i=1}^mR\e_i$ and Let 
 $G = \{\vec{g}_1\dots,\vec{g}_t\} \subset \bigoplus_{i=1}^mR\e_i$ 
 be a $(\Omega,k)$-Gr\"obner basis.
 Then there exists a unique $\vec{r} \in \bigoplus_{i=1}^mR\e_i$ such that 
 $\vec{r}$ is $(\Omega,k)$-reduced with respect to $G$ and 
 \[
  f = \sum_{i=1}^{t}p_i\vec{f}_{i} + \sum_{i=1}^{m}q_i\e_i  + r 
 \]
  where $p_i\in R, i=1,\dots, t$ and 
$q_i \in \langle \Omega(k-|\e_i|) \rangle, i=1,\dots,m$ with
 \[
  \lm(\vec{f}) = \max\left\{\max_{1\leq i \leq t}\lm(p_i\vec{f}_{i}),\;
  \max_{1\leq i \leq m}\lm(q_i\e_i),\;
  \lm(r)
  \right\}.
 \]
\end{proposition} 
The proof is straightforward. See 
\cite[Theorem 1.5.9. and Theorem 1.6.7.]{MR1287608}.
Let $\vec{f}$ and $\vec{g}$ be vectors in $R^m$. The S-vector of
$\vec{f}$ and $\vec{g}$, denoted $S(\vec{f},\vec{g})$, is
\begin{eqnarray*}
  S(\vec{f},\vec{g}):= 
  \frac{\lcm(\lt(\vec{f}),\lt(\vec{g}))}{\lt(\vec{f})}\vec{f} 
- \frac{\lcm(\lt(\vec{f}),\lt(\vec{g}))}{\lt(\vec{g})}\vec{g} .
\end{eqnarray*}
\begin{definition}
 Let $\vec{f} \in M$ be a vector and $g \in R$ be polynomial. 
Set $\lm(\vec{f}) = x_{i_1}x_{i_2}\dots x_{i_k}\e_{i_\vec{f}}$ and 
$\lp(g) = x_{j_1}x_{j_2}\dots x_{j_h}$.
We define
 the coefficient set of S-vectors of $\vec{f}$ with $g$, denoted
 $C(\vec{f};g)$ as
\begin{eqnarray*}
 &&\{(1,x_{i_1}\dots x_{i_\delta},x_{i_{\delta+k+1}}\dots x_{i_{k}}) \in
  R^3 | (i_{\delta+1},\dots,i_{\delta + k}) = (j_1,\dots, j_k)
\}\\
&\cup&
\{(x_{j_1}\dots x_{j_\delta}, 1, x_{i_{h-\delta +1}}\dots x_{i_k}) \in
  R^3 | (i_{1},\dots,i_{h - \delta}) = (j_{\delta+1},\dots, j_h)
\}.
\end{eqnarray*}
\end{definition}
The S-vector of $\g_i$ and $\omega \in \Omega$, 
with regard to $(z,p,q) \in C(\g_i;\omega)$ is, given by
\[
 S(\g_i,\omega;z,p,q) = z\g_i - p\omega q\e_{\mu_i}
\]
where $\mu_i$ represents a position of the non-zero coordinate 
of $\vec{g}_i$.
\begin{theorem}
\label{th:Buchberger_module}
 Let $G = \{\vec{g}_1\dots,\vec{g}_t\}$ be a set of non-zero
 vectors in $\bigoplus_{i=1}^mR\e_i$. Then $G$ is a $(\Omega,k)$-Gr\"obner
 basis for the 
 submodule $M = \langle \vec{g}_1\dots,\vec{g}_t \rangle_{\Omega,k}$ 
 if and only if for all
 $i,j \in \{1,\dots,t\}$
\[
 S(\vec{g}_i, \vec{g}_{j}) \xrightarrow{G}_{\Omega,k} \vec{0},
\]
 and for all  $i \in \{1,\dots,t\}, \omega \in \Omega$ and 
 $(z,p,q) \in C(\vec{g}_i,\omega)$,
\[
 S(\vec{g}_i;z,p,q) \xrightarrow{G}_{\Omega,k} \vec{0}.
\]
\end{theorem}
The proof is basically the same as the one for
Theorem~\ref{th:Buchberger}.

Let $G=\{\vec{g}_1,\dots,\vec{g}_r\}$ be a $(\Omega,k)$-Gr\"obner basis.
For any $\vec{f}\in \bigoplus_{i=1}^mR\e_i$, 
let $\vec{r}$ be a $(\Omega,k)$-reduced vector with respect to $G$
such that $\vec{f} \xrightarrow{G}_+ \vec{r}$. 
Proposition~\ref{prop:division_grobner_module} implies 
that $\vec{r}$ is unique. 
We call $\vec{r}$ the $(\Omega,k)$-normal form of $\vec{f}$ 
with respect to $G$, denoted $N_{G,\Omega,k}(\vec{f})$.

\subsection{Projective resolution}
The proof of Proposition~\ref{prop:syzygy_term} and 
Theorem~\ref{th:Main_theorem} in this section is based on the arguments 
in \cite[Section 3.4.]{MR1287608}.
Let $M$ be a graded projective $R$-module generated by $\e_1, \dots, \e_s$.
 That is,
\[
 M = \bigoplus_{i=1}^mR\e_i.
\]
Let $\rho_{\Omega,k}$ be a natural projection
\[
 \rho_{\Omega,k} \colon \bigoplus_{i=1}^mR\e_i \rightarrow 
 \bigoplus_{i=1}^mR/\langle \Omega(k-|\e_i|) \rangle \e_i
\]
Let $F = \{\vec{f}_1,\dots, \vec{f}_s\}$ be a subset of $M$. We define
the graded $R$-module $N$ as
\[
 N = \bigoplus_{i=1}^sR\e'_j.
\]
with grading $|\e'_j| = |\vec{f}_j|$. We also define the degree
preserving $R$-homomorphism 
$\varphi\colon N \rightarrow M$ by $\varphi_F(\e'_j) = \vec{f}_j$.
We call the kernel of composite $\rho_{\Omega,k} \circ \varphi_F$ the
$(\Omega,k)$-syzygy of $F$, denoted $\Syz_{\Omega,k}(F)$. That is, 
\[
 \Syz_{\Omega,k}(F) = \ker \rho_{\Omega,k} \circ \varphi_F 
                   = \varphi_F^{-1}\left(\bigoplus_{i=1}^m
		   \Omega(k-|\e_i|)\e_i\right).
\]
In this section, we study an algorithm to find a $(\Omega,k)$-generating
set of the submodule $\Syz_{\Omega,k}(F)$ of $N$ for a given finite 
subset $F$ of $M$.
The next proposition shows how to compute these generating set in a
special case.

\begin{proposition}
\label{prop:syzygy_term}
Set $\Lt(\Omega) := \{\lm(\omega) | \omega \in \Omega\}$.
Let $\X_1,\dots, \X_s \in R^m$ be monomials
 of $M$. Set $\X_{i,j} = \lcm(\X_i,\X_j)$. Then 
$\Syz_{\Lt(\Omega),k}(\X_1,\dots,\X_s)$ is $\Lt(\omega),k$-generated by,
\begin{eqnarray*}
LM(\X_1,\dots,\X_s) &:=& \left\{\left.
\frac{\X_{i,j}}{\X_i}\e'_i -
  \frac{\X_{i,j}}{c_j\X_j}\e'_j \right| i,j \in \{1,\dots, s\}
\right\}\\
&\cup&
\left\{
z\e'_i
 \left| i \in \{1,\dots,s\}, z\in R, {}^\exists \lambda \in \Lambda, 
 {}^\exists p,q \in R 
\atop 
  \text{ s.t. } |\omega_\lambda| \leq k,
  (z,p,q) \in C(\X_i;\omega_\lambda) 
  \text{ for some } z \in R.
       \right.
\right\}.
\end{eqnarray*}
\end{proposition}
\begin{proof}
It is easy to see that 
$LM(\X_1,\dots,\X_s) \subset \Syz_{\Lt(\Omega),k}(\X_1,\dots,\X_s)$.

To prove the converse, let
 $(h_1,\dots,h_s) \in \Syz_{\Lt(\Omega),k}(\X_1,\dots,\X_s)$. That is,
\begin{eqnarray}
 \label{eq:1}
  h_1\X_1 + \dots + h_s\X_x = \sum_i p_i\e_i
\end{eqnarray}
for some $p_i \in \langle \Lt(\Omega(k-|\e_{i}|)) \rangle$
Let $\X$ be any
 monomial in $\bigoplus_{i=1}^mR\e_i$. Then the coefficient of $\X$ 
in $h_1\X_1 + \dots + h_s\X_x -
 \sum_{i}p_i\e_i$ must be zero. Thus it is enough to consider
 the case for which $h_i = c_iX'_i$ with $X'_i\X_i = \X$. Let 
$c_{i_1},\dots, c_{i_t}$, with $i_1 < i_2 < \dots < i_t$ be the
 non-zero $c_j$'s. Therefore, we have:
\begin{eqnarray*}
 (h_1,\dots,h_s) &=& (c_1X'_1,\dots, c_sX'_s) = 
  c_{i_1}X'_{i_1}\e'_{i_1} + \dots + c_{i_t}X'_{i_t}\e'_{i_t}\\
  &=& c_{i_1}\frac{\X}{\X_{i_1}}\e'_{i_1} + \dots +
            c_{i_t}\frac{\X}{\X_{i_t}}\e'_{i_t}\\
 &=&
  c_{i_1}\frac{\X}{\X_{i_1i_2}}(\frac{\X_{i_1i_2}}{\X_{i_1}}\e'_{i_1}
  - \frac{\X_{i_1i_2}}{\X_{i_2}}\e'_{i_2})\\
&\phantom{=}& +  (c_{i_1} + c_{i_2})
  \frac{\X}{\X_{i_2i_3}}(\frac{\X_{i_2i_3}}{\X_{i_2}}\e'_{i_2}
  - \frac{\X_{i_2i_3}}{\X_{i_3}}\e'_{i_3}) + \dots\\
&\phantom{=}&  +  (c_{i_1} +\dots +  c_{i_{t-1}})
  \frac{\X}{\X_{i_{t-1}i_t}}(\frac{\X_{i_{t-1}i_t}}{\X_{i_{t-1}}}\e'_{i_{i-1}}
  - \frac{\X_{i_{t-1}i_t}}{\X_{i_t}}\e'_{i_t})\\
&\phantom{=}& (c_{i_1} +\dots +
 c_{i_{t}})\frac{\X}{\X_{i_t}}\e'_{i_t},
\end{eqnarray*}
If $h_1\X_1 + \dots + h_s\X_x = 0$, then we have 
$c_{i_1} +\dots + c_{i_{t}}$ and it follows that 
$(h_1,\dots,h_s) \in \langle LM(\X_1,\dots,\X_s) \rangle$. If not, there
 exists $i\in \{0,\dots m\}$ and 
$\omega \in \Omega(k-|\e'_{i}|)$ such that 
 $\X = X'_{i_t}\X_{i_t} = p\lm(\omega) q\e_i$ where $p$ and $q$ are monomials
 in $R$.
If $C(\X_{i_t},\omega)$ is not empty, then there exists 
$(z,p',q) \in C(\X_{i_t},\omega)$ such that 
$z\X_{i_t} = p'\lm(\omega)q \e_i$. This implies that 
\[
 \frac{\X}{\X_{i_t}}\e_{i_t} = X'_{i_t}\e'_{i_t} = z'z\e'_{i_t}
\]
for some monomial $z'$ in $R$.
If $C(\X_{i_t},\omega)$ is empty, then there exists a monomial $q'$ in
 $R$ such that $X'_{i_t} = p\omega q'$. This implies that
\[
 \frac{\X}{\X_{i_t}}\e_{i_t} = X'_{i_t}\e'_{i_t} = p\omega q'\e'_{i_t}
 \in \bigoplus_{i=1}^s\Lt(\Omega(k-|\e'_{i}|))\e'_i.
\]
Thus we have desired conclusion.
\end{proof}
Let $\{\g_1,\dots, \g_t\}$ be a $(\Omega,k)$-\Groebner basis,
 where we assume that the
$\g_i$'s are monic. For $i\in \{1,\dots,t\}$, we let $\lm(\g_i) = \X_i$ and
for $i \neq j \in \{1,\dots,t\}$, we let $\X_{ij} = \lcm(\X_i,\X_j)$. 
Then the S-polynomial of $\g_i$ and $\g_j$
is, given by
\[
 S(\g_i,\g_j) = \frac{\X{ij}}{\X_i}\g_i - \frac{\X{ij}}{\X_j}\g_j.
\]
We note that $\lm(S(\mathbf{g}_i,\mathbf{g}_j)) < \mathbf{X}_{ij}$.
By Proposition~\ref{prop:division_grobner_module}, we have
\[
 S(\g_i,\g_j) = \sum_{\nu = 1}^{t} u_{ij\nu}\g_\nu 
 + \sum_{\epsilon}p_\epsilon \omega_\epsilon q_\epsilon \e_{i_\epsilon}
\]
where $u_{ij\nu},p_\epsilon, q_\epsilon \in R, \omega_\epsilon \in
 \Omega(k-|\e_{i_\epsilon}|)$, such that
\[
 \max\left\{\max_{1 \leq \nu \leq t}(\lm(u_{ij\nu})\lm(\g_\nu)),
       \max_{\epsilon}(\lm(p_\epsilon \omega_\epsilon
       q_\epsilon)\e_{i_\epsilon}) \right\}
       = \lm(S(\g_i,\g_j)).
\]
We now define 
\[
 \s_{ij} = \frac{\X{ij}}{\X_i}\e'_i - \frac{\X{ij}}{\X_j}\e'_j 
- (u_{ij1},\dots,u_{ijt}) \in R^t.
\]
It is easy to see that $\s_{ij} \in \syz(\g_1,\dots,\g_t)$.

Similarly, the S-polynomial of $\g_i$ and $\omega_\lambda, \lambda\in
\Lambda$, with regard to $(z,p,q) \in C(\g_i;\omega_\lambda )$ is, given by
\[
 S(\g_i,\omega_\lambda;z,p,q) = z\g_i - p\omega_\lambda q\e_{\mu_i}
\]
where $\mu_i$ represents a position of the non-zero coordinate 
of $\vec{g}_i$.
We note that $\lm(S(\g_i,\omega_\lambda;w)) < w$. 
By Proposition~\ref{prop:division_grobner_module}, we have
\[
 S(\g_i,\omega_\lambda;z,p,q) = \sum_{\nu = 1}^{t}h'_{i\nu}\g_\nu
 + \sum_{\epsilon}p_\epsilon \omega_\epsilon q_\epsilon \e_{i_\epsilon}.
\]
for some $h'_{i\nu},p_\epsilon,q_\epsilon \in R$ and 
$\omega_\epsilon \in \Omega(k-|\e_{i_\epsilon}|)$, such that 
\[
\max\left\{\max_{1\leq \nu \leq t}(\lm(h'_{i\nu})\lm(\g_\nu)),
  \max_{\epsilon}(\lm(p_\epsilon \omega_\epsilon
       q_\epsilon)\e_{i_\epsilon}) \right\}
 = \lm(S(\g_i,\omega_\lambda;w)).
\]
We now define 
\[
  s_i(\omega_\lambda;z,p,q) = z\e'_i - (h'_{i1},\dots,h'_{it}).
\]
It is also easy to see that 
$s_i(\omega_\lambda;z,p,q) \in \syz(\g_1,\dots,\g_t)$.
\begin{theorem}
$\Syz_{\Omega,k}(\vec{g}_1,\dots,\vec{g}_t)$ is $(\Omega,k)$-generated by 
\begin{eqnarray*}
M(\g_1,\dots,\g_t) &:=& \{\vec{s}_{i,j}| i,j \in \{1,\dots,t\}\} \\
&\cup&
 \{\vec{s}_{i}(\omega_\lambda;z,p,q) | i \in \{1,\dots,s\}, \lambda \in
 \lambda, |\omega_\lambda| \leq k,(z,p,q) \in C(\g_i,\omega_\lambda)\}.
\end{eqnarray*}
 \label{th:Main_theorem}
\end{theorem}

\begin{proof}
 Suppose to the contrary that there exists 
$(u_1,\dots,u_t) \in \bigoplus_{i=1}^tR\e'_i$ such that
\[
 (u_1,\dots, u_t) \in \syz_{\Omega,k} (\g_1,\dots,\g_t) \setminus 
 \langle M(\g_1,\dots,\g_t) \rangle_{\Omega,k}.
\]
Then we can choose such a $(u_1,\dots, u_t)$ with 
$\X= \max_{1\leq i \leq t}(\lp(u_i)\lm(\g_i))$ least.
Let $S$ be the subset of $\{1,\dots,t\}$ such that
\[
 S = \{i \in \{1,\dots, t\} | \lm(u_i)\lm(\g_i) = \X\}.
\]
Now for each $i \in \{1,\dots,t\}$ we define $u'_i$ as follows:
\[
u'_i = \begin{cases}
	u_i & \text{if } i \notin S\\
	u_i - \lt(u_i) &\text{if } i \in S.
       \end{cases} 
\]
Also, for $i\in S$, let $\lt(u_i) = c_iX'_i$, where $c_i \in K$ and
 $X'_i$ is a monomial in $R$. Since $(u_1,\dots,u_t) \in
 \syz_{\Omega,k}(\g_1,\dots,\g_t)$, we see that
\[
 \sum_{i\in S}c_iX'_i\X_i \in \bigoplus_{i=1}^m\Lt(\Omega(k-|\e_i|))\e_i
\]
and so
\[
 \sum_{i\in S}c_iX'_i\e_i \in  \syz_{\Lt(\Omega),k} (\X_i|i \in S) .
\]
Thus, by Proposition \ref{prop:syzygy_term} we have
\begin{eqnarray*}
  \sum_{i\in S}c_iX'_i\e'_i = \sum_{i<j \; i,j \in S}d_{ij}
     (\frac{\X_{ij}}{\X_i}\e_i - \frac{\X_{ij}}{\X_j}\e_j)
 + \sum_{i\in S,\omega \in \Omega \atop (z,p,q)\in
 C(\g_i,\omega)}b_{i\lambda;z}(z\e'_i)
 + \sum_{\epsilon, i\in S}p_{\epsilon i}\lm(\omega_{\epsilon i})
 q_{\epsilon i}\e'_{i}
\end{eqnarray*}
for some monomials $d_{ij}, b_{i\lambda;z}, p_{\epsilon i}, q_{\epsilon i} \in R$
 and $\omega_{\epsilon i} \in \Omega$.
Since each coordinate of the vector in the
 left-hand side of the equation above is homogeneous, and since $X'_i\X_i
 = \X$, we can choose $d_{ij}$ to be a constant multiple of
 $\frac{\X}{\X_{ij}}$. Similarly, we can choose $b_{i\lambda;z}$ to
 be a constant multiple of $\frac{\X}{z\X_{i}}$. Set 
$\bar{\omega}_{\epsilon i} := \omega_{\epsilon i} - \lm(\omega_{\epsilon i})$.
Then we have
\begin{eqnarray*}
 (u_1,\dots,u_t) &=&  \sum_{i\in S}c_iX'_i\e_i +  (u'_1,\dots,u'_t) \\
       &=&\sum_{i<j \; i,j \in S}
	d_{ij}\left(\frac{\X_{ij}}{\X_i}\e_i - \frac{\X_{ij}}{\X_j}\e_j\right)
 + \sum_{i\in S,\omega \in \Omega\atop (z,p,q)\in C(\g_i,\omega)}
      b_{i\lambda;z}(z\e'_i) \\
 &+& \sum_{\epsilon, i\in S}p_{\epsilon i}\lm(\omega_{\epsilon i})
    q_{\epsilon i}\e'_{i}
 + (u'_1,\dots,u'_t)     \\
 &=&\sum_{i<j \; i,j \in S}d_{ij}\s_{ij} +
    \sum_{i\in S,\omega \in \Omega\atop (z,p,q)\in C(\g_i,\omega)}
         b_{i\lambda;z}\s_i(\omega;z,p,q) \\
 &+& (u'_1,\dots,u'_t) +
	\sum_{i<j \; i,j \in S}d_{ij}(h_{ij1},\dots,h_{ijt})
     +\sum_{i\in S,\lambda,\atop (z,p,q)\in
           C(\g_i,\omega)}b_{i\lambda;z}(h'_{i1},\dots,h'_{it})\\
 &+& \sum_{\epsilon, i\in S} 
        p_{\epsilon,i}\omega_{\epsilon i} q_{\epsilon i}\e'_{i} 
  -  \sum_{\epsilon, i\in S}
        p_{\epsilon i}\bar{\omega}_{\epsilon i} q_{\epsilon i}\e'_{i}.
\end{eqnarray*}
We define 
\begin{eqnarray*}
  (v_1,\dots,v_t) &:=& (u'_1,\dots,u'_t) +
	\sum_{i<j \;  i,j \in S}d_{ij}(h_{ij1},\dots,h_{ijt}) +
  \sum_{i\in S,\omega \in \Omega\atop (z,p,q)\in C(\g_i,\omega)}
  b_{i\lambda;z}(h'_{i1},\dots,h'_{it}) \\
&-& \sum_{\epsilon, i\in S}
        p_{\epsilon i}\bar{\omega}_{\epsilon i} q_{\epsilon i}\e'_{i}.
\end{eqnarray*}
We note that 
$(v_1,\dots,v_t) \in \syz_{\Omega,k}(\g_1,\dots,\g_t) 
\setminus \langle M(\g_1,\dots,\g_t) \rangle_{\Omega,k}$, 
since 
$(u_1,\dots,u_t), \s_{ij}, \s_i(\omega;z,p,q) \in 
\syz_{\Omega,k}(\g_1,\dots,\g_t)$ 
and 
$(u_1,\dots, u_t) \notin \langle M(\g_1,\dots,\g_t) \rangle_{\Omega,k}$. 
We will obtain the desired contradiction by proving that $\max_{1 \leq
 \nu \leq t} (\lm(v_\nu)\lm(\g_\nu)) < \X$. For each $\nu \in \{1,\dots,
 t\}$
we have 
\begin{eqnarray*}
 \lm(v_\nu)\lm(\g_\nu) 
 &=& \lm\left(u'_\nu + \sum_{i<j \;  i,j \in S} d_{ij}\lm(h_{ij\nu}) 
     + \sum_{i\in S,\omega \in \Omega \atop (z,p,q)\in C(\g_i,\omega)}
        b_{i\lambda;z}\lm(h'_{i\nu}) 
	+ \sum_{\epsilon} 
	    p_{\epsilon \nu}\bar{\omega}_{\epsilon \nu} q_{\epsilon \nu}
     \right)\X_\nu\\
 &\leq& \max\left(\lm(u'_\nu), 
	\max_{i<j \;  i,j \in S}d_{ij}\lm(h_{ij\nu}),
	\max_{i\in S,\omega \in \Omega \atop (z,p,q)\in C(\g_i,\omega)}
	b_{i\lambda;z}\lm(h'_{i\nu}),
	\max_{\epsilon}p_{\epsilon \nu}\bar{\omega}_{\epsilon \nu} q_{\epsilon \nu},
  \right)\X_\nu
\end{eqnarray*}
where we assume that 
$p_{\epsilon \nu}\bar{\omega}_{\epsilon \nu} q_{\epsilon \nu} = 0$ if
 $\nu \notin S$.
However, by definition of $u'_\nu$, we have $\lm(u'_\nu)\X_\nu <
 X$. Also, as mentioned above, $d_{ij}$ is a constant multiple of
 $\frac{\X}{\X_{ij}}$, and hence for all $i,j \in S, i<j$, we have
\[
 d_{ij}\lm(h_{ij\nu})\X_\nu =  \frac{\X}{\X_{ij}}\lm(h_{ij\nu})\X_\nu
 \leq  \frac{\X}{\X_{ij}}\lm(\s(\g_i,\g_j)) 
<  \frac{\X}{\X_{ij}}\X_{ij} = \X.
\]
Similarly, $b_{i\lambda;z}$ is a constant multiple of
 $\frac{\X}{z\X_i}$, we have
\[
 b_{i\lambda;z}\lm(h'_{i\nu})\X_\nu = \frac{\X}{z\X_i}\lm(h'_{i\nu})\X_\nu
\leq \frac{\X}{z\X_i}\lm(\s_i(\omega;z,p,q)) 
 < \frac{\X}{z\X_i} z\X_i = \X.
\]
Also, by the definition of $\bar{\omega}_{\epsilon \nu}$,
\[
 \lm(p_{\epsilon \nu} \bar{\omega}_{\epsilon \nu} q_{\epsilon \nu}) \X_\nu < 
 \lm(p_{\epsilon \nu}) \lm(\omega_{\epsilon \nu}) \lm(q_{\epsilon \nu}) \X_\nu
 = \X.
\]
Therefore $ \lm(v_{\nu})\lm(\g_{\nu})<\X$ for each $\nu \in
\{1,\dots, t\}$ violating the condition that $\X = \max_{1\leq \nu \leq
 t}(\lm(u_\nu)\lm(\g_\nu))$ is least.
\end{proof}
Finally we study the algorithm to compute 
$\Syz_{\Omega,k}(\vec{f}_1,\dots,\vec{f}_s)$ for 
$\{\vec{f}_1,\dots,\vec{f}_s\}$ be a collection of $(\Omega,k)$-reduced
vectors in $M =\bigoplus_{i=1}^uR\e_{i}$ which may not form a Gr\"obner
basis.
First we compute an $(\Omega,k)$-Gr\"obner basis
$\{\vec{g}_1,\dots,\vec{g}_t\}$.  We again assume that
$\g_1,\dots,\g_t$ are monic. Let 
\begin{eqnarray*}
 F = \left[\begin{array}[]{c}
            \vec{f}_1\\
	    \vdots\\
	    \vec{f}_s\\
	   \end{array}\right]
\text{ and }
 G = \left[\begin{array}[]{c}
            \vec{g}_1\\
	    \vdots\\
	    \vec{g}_t\\
	   \end{array}\right]
\end{eqnarray*}
be non-zero matrix of row vectors
in $M$. 
There exists an $s \times t$ matrix $S$ and a $t\times s$ matrix $T$ with
entries in $R$ such that $F=N_{\Omega,k}(SG)$ and $G=N_{\Omega,k}(TF)$. 
Using Theorem~\ref{th:Main_theorem}, we can compute a generating
set $\{\vec{s}_1,\dots,\vec{s}_r\}$ for $\syz_{\Omega,k}(G)$.
Therefore for each $i=1,\dots,r$
\[
 \vec{0} =  N_{\Omega,k}(\s_iG) = N_{\Omega,k}(\s_iTF) = N_{\Omega,k}((\s_iT)F )
\]
and hence
\[
 \langle \s_iT | i = 1,\dots,r \rangle \subset \syz_{\Omega,k}(F).
\]
Moreover, if we let $I_s$ be the $s \times s$ identity matrix, we have
\[
  N_{\Omega,k}((I_s-ST)F) = N_{\Omega,k}(F - STF) = 
  N_{\Omega,k}(F - SG) = \mathbf{0} 
\]
and hence the rows $\mathbf{r}_1,\dots, \mathbf{r}_s$ of $I_s - TS$ 
are also in $\syz_{\Omega,k}(F)$.
\begin{theorem}
\label{th:syzygy_non-groebner}
 With the notation above we have
\[
 \syz_{\Omega,k}(\mathbf{f}_1,\dots,\mathbf{f}_s)=\langle
 \s_1T,\dots,\s_rT,\mathbf{r}_1,\dots,\mathbf{r}_s \rangle
\]
\end{theorem}
The proof is straightforward and same as that of 
\cite[Theorem 3.4.3]{MR1287608}.

To compute $\Ext$ efficiently, we need to find a generating set of
the syzygy which has a small cardinality.
\begin{definition}
\label{def:minimal_generating_set}
Let $M$ be a submodule of $\bigoplus_{i=1}^sR\e_i$.
 $\{\vec{f}_1,\dots,\vec{f}_r\} \subset \bigoplus_{i=1}^sR\e_i$ is
 $(\Omega,k)$-minimal generating set of $M$ if 
$M = \langle \vec{f}_1,\dots,\vec{f}_r \rangle_{\Omega,k}$ and 
\[
 \vec{f}_j \notin \langle \vec{f}_1,\dots,\vec{f}_{j-1},
            \vec{f}_{j+1},\dots,\vec{f}_r \rangle_{\Omega,k}
\]
   for any $j=1,\dots,r$.
\end{definition}
\begin{proposition}
\label{prop:minimal_algorithm}
Let $\vec{h}_1,\dots \vec{h}_r$ be vectors in $\bigoplus_{i=1}^sR\e_i$.
The following algorithm produces a $(\Omega,k)$-minimal
 generating set of $\langle \vec{h}_1,\dots \vec{h}_r\rangle_{\Omega,k}$.
\begin{algorithm}
\caption{title}
 \label{algo:minimal_generating_set}
\begin{algorithmic}
\FOR{$i=0$ to $r$}
    \STATE $G \leftarrow $Gr\"obner basis of 
$\vec{h}_1,\dots, \vec{h}_{i-1},\vec{h}_{i+1},\dots \vec{h}_r$
    \STATE $\vec{h}_i \leftarrow N_{G,\Omega,k}(\vec{h}_i)$
\ENDFOR
\STATE Let $i_1,\dots i_{r'}$ with $i_1\leq \dots \leq i_{r'}$ be 
non-zero $\vec{h}_i's$.
\RETURN $\vec{h}_{i_1} \dots \vec{h}_{i_{r'}}$
\end{algorithmic}
\end{algorithm}
\end{proposition}
\begin{remark0}
 In general, the output of Algorithm~\ref{algo:minimal_generating_set} 
 is not a Gr\"obner basis even if the input is.
\end{remark0}
\begin{proposition}
\label{prop:minimal}
Let $\vec{f}_1,\dots,\vec{f}_r$ be homogeneous vectors in
 $\bigoplus_{i=1}^sR\e_i$.
If $\{\vec{f}_1,\dots,\vec{f}_r\}$ is $(\Omega,k)$-minimal generating
 set, then
\[
 \syz_{\Omega,k}(\vec{f}_1,\dots,\vec{f}_s) \subset 
 \bigoplus_{i=1}^rI(R)\e'_i.
\]
\end{proposition}
\begin{proof}
 Assume that there exists $\vec{v} = (v_1,\dots,v_r) \in 
\syz_{\Omega,k}(\vec{f}_1,\dots,\vec{f}_s)$ such that 
$\vec{v} \notin  \bigoplus_{i=1}^rI(R)\e'_i$. Then we have 
$v_i \notin I(R)$ for some $i \in \{1,\dots,r\}$. We can suppose that
$v_i$ is homogeneous. Thus we have $v_i \in K\setminus \{0\}$. 
By the assumption of $\vec{v}$, 
$\sum v_i\vec{f}_i \in \bigoplus_{i=1}^s\Omega(k-|\e_i|)\e_i$.
It follows that $\vec{f}_j \in \langle  \vec{f}_1,\dots,\vec{f}_{i-1},
            \vec{f}_{i+1},\dots,\vec{f}_r \rangle_{\Omega,k}$.
This contradicts 
that $\{\vec{f}_1,\dots,\vec{f}_r\}$ is $(\Omega,k)$-minimal generating
set.
\end{proof}
\section{Projective resolution over the noncommutative ring}
\label{sec:proj-resol-over}
In this section, we explain an algorithm to compute a minimal projective
resolution.
Let $K,R,\Omega,k$ be as above and set $\Gamma := R/\langle \Omega \rangle$.
Suppose that we have a degree preserving $\Gamma$ homomorphism $d_n$:
\[
d_n\colon \bigoplus_{i=1}^{t} \Gamma\e'_i 
            \rightarrow \bigoplus_{i=j}^{s} \Gamma\e_j.
\]
Set $\bar{\vec{f}}_i = d_n(\e'_i)$.
We can choose $\{\vec{f}_1,\dots,\vec{f}_t\} \subset \bigoplus_{i=1}^s R\e_i$
such that
\[
 \rho_{\Omega,k}(\vec{f}_i) = \bar{\vec{f}}_i 
 \in \bigoplus_{i=1}^sR/\langle \Omega(k-|\e_i|)\e_i\rangle
 \cong_k \bigoplus_{i=1}^s\Gamma\e_i.
\]
Here we also suppose that $\{\vec{f}_1\dots \vec{f}_M\}$ is 
$(\Omega,k)$-minimal generating set.
Then we have a degree-preserving $R$-homomorphism
\[
 \tilde{d}_n \colon \bigoplus_{i=1}^t R\e_i \rightarrow \bigoplus_{i=1}^s R\e'_i
\]
defined by $\tilde{d}_n(\e'_i) = \vec{f}_i$.
Using Theorem~\ref{th:syzygy_non-groebner} and
Proposition~\ref{prop:minimal_algorithm}, we can compute an 
$(\Omega,k)$-minimal generating set $\{\vec{h}_1,\dots, \vec{h}_r\}$ of 
$\syz_{\Omega,k}(\vec{f}_1,\dots,\vec{f}_t)$. Let $\bar{\vec{h}}_i$ be
the image of $\vec{h}_i$ by the projection
\[
\bigoplus_{i=1}^t R\e'_i \rightarrow  
          \bigoplus_{i=1}^t R/\langle \Omega(k-|\e'_i|)\e'_i\rangle
	  \cong_k \bigoplus_{i=1}^t\Gamma\e'_i.
\]
Then we have a sequence of minimal $\Gamma$-homomorphisms,
which is exact in degree less than or equal to $k$,
\[
 \bigoplus_{i=1}^r \Gamma\e''_i \xrightarrow{d_{n+1}}
 \bigoplus_{i=1}^t \Gamma\e'_i \xrightarrow{d_{n}}
 \bigoplus_{i=1}^s \Gamma\e_i
\]
where $|\e''_i| = |\bar{\vec{h}}_i|$ and
 $d_{n+1}(\e''_i) = \bar{\vec{h}}_i$.
Moreover, it follows from Proposition~\ref{prop:minimal} that 
\[
 d_{n+1}\left( \bigoplus_{i=1}^r \Gamma\e''_i\right) \subset 
      I(\Gamma) \left(\bigoplus_{i=1}^t \Gamma\e'_i\right).
\]


\section*{Appendix}
\subsection*{Steenrod algebra and Gr\"{o}bner basis}
For details of the Steenrod algebra, see \cite[chapter 4]{MR1793722}.
Let $R = \F_2 \langle \Sq^1, \Sq^2, \cdots \rangle$ be a free
noncommutative graded algebra with grading $|\Sq^i| = \deg \Sq^i = i$. 

We choose a ordering on the set of monomials of $R$ to be that of
Example~\ref{ex:lex-order}.
Let $\Omega_{\mathrm{Adem}}$ 
be a subset of $R$ consisted of Adem relations,
\[
 \omega(a,b) := \Sq^a\Sq^b - \sum_{j = 0}^{[a/2]}\binom{b-1-j}{a-2j}\Sq^{a+b-j}\Sq^j
\]
for $0 < a < b$. 
We define a monomial ordering on the set of monomials of $R$ as follows:
Let $ \Sq^{a_1}\cdots \Sq^{a_k}$ and $Y = \Sq^{b_1}\cdots \Sq^{b_l}$ 
be monomials of $R$.
Then, $X \geq Y$ if $k > l$ or,
$k = l$ and the {\it right-most} nonzero entry of $(a_1 - b_1,\dots,a_k - b_k)$
is positive. It follows that $\lm(\omega(a,b)) = \Sq^a\Sq^b$.
The Steenrod algebra is given by a quotient:
\[
 {\A_2} \cong R/ \langle \Omega_{\mathrm{Adem}} \rangle 
\]
where $\langle \Omega_{\mathrm{Adem}} \rangle$ denotes the two-side ideal of
$R$ generated by $\Omega_{\mathrm{Adem}}$.
\begin{definition}
 A sequence $I = (a_1,\dots,a_n) \subset \mathbb{N}$ is called an admissible
 sequence if $a_1,\dots,a_n$ satisfies $a_i \geq 2a_{i+1}$ for
 $i=1,\dots,n-1$.
 $\Sq^I = \Sq^{a_1}\cdots \Sq^{a_n}$ is called an admissible element if
 $(a_1,\dots,a_n)$ is an admissible sequence.
\end{definition}
Considering the action of the Steenrod algebra on the
$\Z/2$-cohomology of $\mathbb{R}\mathrm{P}^\infty$, it follows that 
Admissible elements form a basis of $\F_2$-vector space 
${\A_2} = R/\langle \Omega_{\mathrm{Adem}} \rangle$ 
\cite[Theorem 4.46]{MR1793722}.
By Theorem~\ref{th:equivalent_conditions}, we have
\begin{proposition}
 $\Omega_{\mathrm{Adem}}$ is a Gr\"obner basis.
\end{proposition}

\bibliographystyle{amsplain}
\bibliography{/DirUsers/tomo_xi/Library/tex/books,/DirUsers/tomo_xi/Library/tex/math}

\bigskip
\address{ 
Department of Mathematics\\
Kyoto University \\
Kyoto 606-8502\\
Japan
}
{tomo\_xi@math.kyoto-u.ac.jp}
\end{document}